\documentclass[leqno]{article}
\usepackage{Nov2019}
\newcommand{\myreferences}{/users/harpreet/Dropbox/Biblio/bibThesis}

\begin{document}

\title{$\deg \ds{Q}$ Algebraic Geometry}
\author{Harpreet Singh Bedi~~~{bedi@alfred.edu}}
\date{15 Aug 2019}
\maketitle
\begin{abstract}
Elementary Algebraic Geometry can be described as study of zeros of polynomials with integer degrees, this idea can be naturally carried over to `polynomials' with rational degree.
 This paper explores affine varieties, tangent space and projective space for such polynomials and notes the differences and similarities between rational and integer degrees. The line bundles $\curly{O}(n),n\in\ds{Q}$ are also constructed and their \v{C}ech cohomology computed.
\end{abstract}
\tableofcontents
\pagebreak

The study of algebraic geometry begins with zeros of polynomials in a ring $k[X]$. In this tract it is shown that the ideas of algebraic geometry can be naturally carried over to polynomials of rational degree, for example $X^{1/2}-Y^2=0$. The ring of rational degree polynomials will be denoted as $k[X]_\ds{Q}$, where $k$ is algebraically closed and is thus of infinite cardinality.

In section \ref{rat-deg} the polynomial rings with rational degree are constructed via direct limit and affine algebraic sets are defined in section \ref{affine1}. These rings are Non Noetherian but they are B\'ezout and coherent as shown in lemma \ref{bezout} and proposition \ref{prop-coherent}.
\begin{lemma*}
  The ring $k[T]_{\ds{Q}}$ is a B\'ezout Domain.
\end{lemma*}
\begin{proposition*}
  Let $R$ be a noetherian ring then $R[T_1,\ldots,T_n]_{\ds{Q}}$ is coherent.
\end{proposition*}

The Nullstellensatz is proven in theorem \ref{nulstell}.

\begin{theorem*} Let $k$ be algebraically closed and uncountable and  $I\subsetneq k[T_1,\ldots,T_n]_{\ds{Q}}$ be a proper ideal.
 Then $V(I)$ is non empty.
\end{theorem*}

The lemmas and results corresponding to Noether Normalization are proved in section \ref{normalization}. Finite fields with degree $\ds{Z}[1/p]_{\geq 0}$ are introduced in section \ref{finite-field} and the following result for finite fields is proved in theorem \ref{finite-field-thm1}.

\begin{theorem*}
  Let $V\subset \ds{A}^n,W\subset A^m$ and $U\subset A^\ell$ be algebraic sets as defined in section \ref{finite-field}.
  \begin{enumerate}
     \item A polynomial map $\phi: V \ra W $ induces a $k$-algebra homomorphism $\phi^* : k[W]_{\ds{Z}[1/p]} \ra k[V]_{\ds{Z}[1/p]}$, defined by
composition of functions; that is, if $g \in k[W]_{\ds{Z}[1/p]}$ is a polynomial function then so is $\phi^* (g) = g \circ \phi$.
\item  If $\phi:V\ra W$ and $\varphi:W\ra U$ are polynomial maps then $(\varphi\circ\phi)^*=\phi^*\circ\varphi^*$.
\item If $\Phi:k[W]\ra k[V]$  then it is of the form $\Phi=\phi^*$ where $\phi:V\ra W$ is a unique polynomial map.
  \end{enumerate}
\end{theorem*}

Projective geometry begins in section \ref{proj-geo}, the line bundles $\curly{O}(n),n\in\ds{Q}$ are constructed in section \ref{sec-twist}, and their cohomology computed in theorem \ref{t1} and \ref{t2}.
\begin{theorem*}
  Let $S=R[X_0,\ldots, X_n]_{\ds{Q}}$ and $X=\Proj~S$, then for any $n\in\ds{Q}$
\begin{enumerate}
\item There is an isomorphism $S\simeq \oplus_{n\in\ds{Q}} H^0(X,\mathscr{O}_X(n))$.
\item $H^n(X,\curly{O}_{X}(-n-1))$ is a free module of infinite rank.
\end{enumerate}
\end{theorem*}

\begin{theorem*} Let $S=R[X_0,\ldots, X_n]_{\ds{Q}}$ and $X=\Proj~R$, then
$H^i(X,\curly{O}_X(m))=0$ if $0<i<n$.
\end{theorem*}

Tangent spaces are constructed in section \ref{tangent-space}. The applications are given in section \ref{sec-app}, these include embeddings of degree one, perfectoid algebras and fractional blow up algebras.


\section{Rational Degree}\label{rat-deg} The ring of polynomials with rational degrees can be constructed by formally attaching $X^{1/i}, i\in\ds{Z}_{>0}$ to the ring $k[X]$ and is denoted as $k[X]_{\ds{Q}}$. Since, all possible denominators $i\in\ds{Z}_{>0}$ are attached, all $X^{a/b} a, b\in\ds{Z}_{>0}$ lie in the ring $k[X]_{\ds{Q}}$.

It is also possible to use direct limit for construction of $k[X]_{\ds{Q}}$. The advantage of this construction is the fact that direct limit is an exact functor, thus short exact sequences in $k[X]$ can be carried over to $k[X]_{\ds{Q}}$.
\subsection{Rational degree via Direct Limit}
\begin{construction}
  Consider the following inclusions
  \begin{equation}\label{dir-lim-1-var}
    k[X]\subseteq k[X,X^{1/2}]\subseteq k[X,X^{1/2},X^{1/3}]\subseteq\ldots\subseteq\varinjlim_i k[X,\ldots, X^{1/i}]=\bigcup_i k[X,\ldots, X^{1/i}]
  \end{equation}
  The above construction can be carried onto multiple variables
  \begin{equation}
    \begin{aligned}
      k[X_1,\ldots, X_n]\subseteq k[X_1,X_1^{1/2},\ldots, X_n, X_n^{1/2}]\subseteq k[X_1,X_1^{1/2},X_1^{1/3},\ldots,X_n,X_n^{1/2},X_n^{1/3}]\subseteq\ldots\\
      \ldots\subseteq\varinjlim_i k[X_1,\ldots, X_1^{1/i},\ldots, X_n\ldots, X_n^{1/i}]=\bigcup_i k[X_1,\ldots, X_1^{1/i},\ldots,X_n,\ldots, X_n^{1/i} ]
    \end{aligned}
  \end{equation}

\end{construction}

\begin{define} The following notation will be used throughout
  \begin{equation}
    \begin{aligned}
      k[X]_\ds{Q}&=\bigcup_i k[X,\ldots, X^{1/i}]\\
      k[X_1,\ldots, X_n]_\ds{Q}&=\bigcup_i k[X_1,\ldots, X_1^{1/i},\ldots,X_n,\ldots, X_n^{1/i} ]
    \end{aligned}
  \end{equation}

\end{define}
Further constructions for multivariate case are as follows:

\begin{construction}

  \begin{enumerate}
    \item The multivariate case can also be obtained inductively. First construct the ring $k[X]_{\ds{Q}}$ and denote it as $R$ and attach attach $Y^{1/i}, i\in\ds{Z}_{>0}$ as in \eqref{dir-lim-1-var} to get
    \[R[Y]\subseteq R[Y,Y^{1/2}]\subseteq \ldots\subseteq\varinjlim_i R[Y,\ldots, Y^{1/i}]=\bigcup_i R[Y,\ldots, Y^{1/i}]=k[X,Y]_{\ds{Q}}.\]

    \item Construction via tensor products \begin{equation}
      k[X_1,\ldots,X_n]_\ds{Q}=k[X_1]_{\ds{Q}}\otimes_k\cdots\otimes_kk[X_n]_{\ds{Q}},
    \end{equation}
    where each $k[X_i]_{\ds{Q}}$ is constructed as in \eqref{dir-lim-1-var}. In fact, the above construction is an application of direct limit to tensor product and the observation that $\varinjlim (A\otimes B)=\varinjlim A\otimes\varinjlim B$.
  \[k[X,Y]_\ds{Q}=\varinjlim_i k[X,\ldots, X^{1/i}]\otimes_k \varinjlim_i k[Y,\ldots, Y^{1/i}]=\varinjlim_i k[X,\ldots, X^{1/i},Y,\ldots, Y^{1/i} ]\]

  \end{enumerate}

\end{construction}

\begin{lemma}\label{lem-change-var}
  For each $i$ the following inclusions hold
\begin{equation}
  k[X,\ldots,X^{1/i}]\subseteq k[Y]\text{ and }k[X_1,\ldots,X_1^{1/i}, \ldots, X_n,\ldots,X_n^{1/i}]\subseteq k[Y_1,\ldots, Y_n]
\end{equation}

\end{lemma}
\begin{proof}
  Let $\ell=$ least common multiple of $(1,\ldots, i)$ then $X^{1/\ell}\mapsto Y$ gives the first inclusion. The second inclusion is obtained from $X_j^{1/\ell}\mapsto Y_j$ for $1\leq j\leq n$.
\end{proof}

\begin{lemma}\label{finite-zero}
  Let $f\in k[T]_{\ds{Q}}$, then $f$ has only finitely many zeros.
\end{lemma}
\begin{proof}
  It is well known that a polynomial with integer degree has finitely many zeros. The idea is to convert any rational degree polynomial into an integer degree polynomial by substitution.
Find the least common multiple of the denominators of rational degree of $T$ and denote it by $d$. Make the substitution $T^{1/d}\mapsto X$. This will give a polynomial in $k[X]$ with integer degrees and solution say $X=a$, (since $k$ is assumed to be algebraically closed) which gives $T=a^d$. Thus, giving finitely many roots for $f$.

\end{proof}

Recall that  B\'ezout domain is an integral domain in which every finitely generated ideal is principal.

\begin{lemma}\label{bezout}
  The ring $k[T]_{\ds{Q}}$ is a B\'ezout Domain.
\end{lemma}
\begin{proof}
First, notice that $k[T]_{\ds{Q}}$ is not a principal ideal domain, since the ideal $\lr{X,X^{1/2},X^{1/3},\ldots}$ cannot be generated by a single element. Let $I$ be a finitely generated ideal of $k[T]_{\ds{Q}}=\bigcup_i k[T^{1/i}]$, the finiteness of $I$ forces it to lie in one of the rings $k[T^{1/i}]$ which by a change of variable $T^{1/i}\mapsto X$ is a principal ideal domain $k[X]$, thus the ideal $I$ is generated by a single element.
\end{proof}

\begin{proposition}\label{prop-coherent}
  Let $R$ be a noetherian ring then $R[T_1,\ldots,T_n]_{\ds{Q}}$ is coherent.
\end{proposition}
\begin{proof}
Let $\id{J}=(f_1,\ldots, f_r)$ be a finitely generated ideal in the ring, then finite generation of $\id{J}$ forces it to lie in the ring $B_i$ for $i$ large enough, where
\begin{equation}
  B_i=R[T_1,\ldots,T_1^{1/i},\ldots,T_n,\ldots,T_n^{1/i}] \quad\text{ and }\quad\varinjlim_iB_i=R[T_1,\ldots,T_n]_{\ds{Q}}.
\end{equation}
Since, $B_i$ is Noetherian there is a finite presentation of the form
\begin{equation}
  B_i^{m_1}\ra B_i^{m_2}\ra\id{J}\ra 0.
\end{equation}
Applying the direct limit to the above gives the required result.
\begin{equation}
  (R[T_1,\ldots,T_n]_{\ds{Q}})^{m_1}\ra (R[T_1,\ldots,T_n]_{\ds{Q}})^{m_2}\ra\id{J}\ra 0.
\end{equation}
\end{proof}

\begin{remark}
  The above lemmas show that each of the ring $k[X,\ldots,X^{1/i}]$ is noetherian but the ring $k[X]_\ds{Q}$ is non-noetherian. Consider the non terminating chain of ideals $(X)\subsetneq (X^{1/2})\subsetneq\ldots\subsetneq(X^{1/2^i})\subsetneq\ldots $
\end{remark}

\section{Affine Algebraic Sets}\label{affine1}

Fix a field $k$ and an integer $n$ and let $A:=k[T_1,\ldots, T_n]_{\ds{Q}}$ be the ring of polynomials with rational degrees in $n$ indeterminates and coefficients in $k$.
If $F(T_1,\ldots,T_n)\in A$, a point $x=(x_1,\ldots, x_n)\in k^n$ is a zero of $F$ if $F(x_1,\ldots,x_n)=0$.

\begin{define}
Let $S$ be a subset of $k[T_1,\ldots, T_n]_{\ds{Q}}$ and let $V(S)$ denote the subset of $k^n$ formed by common zeros of all elements of $S$. The subsets of $k^n$ of this type will be called affine algebraic set defined by $S$.
\begin{equation}
  V(S)=\{x\in k^n\text{ such that }F(x)=0\text{ for all }F\in S\}.
\end{equation}
\end{define}

The affine algebraic sets form the closed sets of the Zariski Topology on $k^n$, some of its properties are listed below.

\begin{remark}Let $A=k[T_1,\ldots,T_n]_\ds{Q}$
\begin{enumerate}
    \item $V$ is inclusion reversing, if $S_1\subseteq S_2\subseteq A$ then $V(S_1)\supseteq V(S_2)$.
  \item Multiple polynomials can define the same affine algebraic set. For example, $V(T^2)=V(T)=V(T^{1/2})$.
  \end{enumerate}
\end{remark}

\begin{define}
  Let $\id{a}$ be an ideal of $k[T_1,\ldots, T_n]_{\ds{Q}}$, then the set $V(\id{a})$ will be called the closed sets of $k^n$.
\end{define}

\begin{proposition}The closed sets define the Zariski topology.
\end{proposition}
\begin{proof} Let $A=k[T_1,\ldots,T_n]_\ds{Q}$
  \begin{enumerate}
    \item $V(0)=k^n$ and $V(1)=\varnothing$.
    \item $V(\id{a}\id{b})=V(\id{a}\cap\id{b})=V(\id{a})\cup V(\id{b})$ for $\id{a},\id{b}\in A$.\\
     Since, $\id{ab}\subseteq \id{a}\cap\id{b}\subseteq \id{a,b}$ applying $V(-)$ reverses the inclusion $V(\id{ab})\supseteq V(\id{a}\cap\id{b})\supseteq V(\id{a})\cup V(\id{b}).$ On the other hand, if $x\notin V(\id{a})$ and $x\in V(\id{a}\id{b})$, there exists $a\in\id{a}$ such that $a(x)\neq 0$ and for all $b\in\id{b}, a(x)b(x)=0$ implying $b(x)=0$ or $x\in V(\id{b})$.
    \item For a family of ideal $\id{a}_i\in A$ with $i\in I$ the following holds
    \begin{equation}
      \bigcap_{i\in I}V(\id{a}_i)=V\left(\bigcup_{i\in I} a_i\right)=V\left(\sum_{i\in I}\id{a}_i\right)
    \end{equation}
  \end{enumerate}
\end{proof}

\begin{remark}\label{maximal-ideal}
Recall that $\ds{C}[X]/(X-a)=\ds{C}$ for $a\in\ds{C}$, this can be re-written as
\begin{equation}
\frac{\ds{C}[X]}{(X-a),(X^2-a^2),\cdots ,(X^i-a^i),\cdots}  =\frac{\ds{C}[X]}{X-a}=\ds{C}\text{ for }i\in\ds{Z}_{>0}.
\end{equation}
The first equality comes from the fact that $(X-a)\supset (X^2-a^2)\supset\cdots \supset(X^i-a^i)\supset\cdots $. The above can be adapted for $\ds{C}[X]_\ds{Q}$ as
\begin{equation}
  \frac{\ds{C}[X]_\ds{Q}}{(X-a),\ldots, (X^{1/i}-a^{1/i}),\ldots}=\ds{C}\text{ for }i\in\ds{Z}_{>0},
\end{equation}
giving an example of a infinitely generated maximal ideal in the ring $\ds{C}[X]_{\ds{Q}}$. The same argument works for any algebraically closed field $k$. Hence, for any evaluation map at
$a\in k$ the associated maximal ideal in $k[X]_{\ds{Q}}$ is generated by the family $\{(X^{1/i}-a^{1/i})\}_{i\in\ds{Z}_{>0}}$. Similarly, the maximal ideal associated to evaluation at $(a_1,\ldots,a_n)\in k^n$ in the ring $k[T_1,\ldots, T_n]_{\ds{Q}}$ is generated by the family $\{(T_1^{1/i}-a_1^{1/i}),\ldots,(T_n^{1/i}-a_n^{1/i}) \}_{i\in\ds{Z}_{>0}}$.

For example, in the ring $k[X]_{\ds{Q}}$ the evaluation map at $a=0$ has the corresponding maximal ideal as
\begin{equation}
  \id{q}:=\lr{X,X^{1/2},X^{1/3},\ldots}=X^{1/i}_{i\in\ds{Z}_{>0}}\qquad{\text{and note that}}\qquad \id{q}^2=\id{q}.
\end{equation}

\textdbend It is necessary to have $k$ as algebraically closed, so that all roots of $a\in k$ can be obtained by solving equations of the form $X^i-a=0,i\in\ds{Z}_{>0}$. For example $2\in\ds{Q}$ and the corresponding ideal would be  $\{(X^{1/i}-2^{1/i})\}_{i\in\ds{Z}_{>0}}$ in the ring $\ds{Q}[X]_{\ds{Q}}$. But this gives
\begin{equation}
  \frac{\ds{Q}[X]_{\ds{Q}}}{\{(X^{1/i}-2^{1/i})\}_{i\in\ds{Z}_{>0}}}=\ds{Q}[2^{1/i}]_{i\in\ds{Z}_{>0}}\supsetneq \ds{Q}.
\end{equation}
Notice that $\ds{Q}[2^{1/i}]_{i\in\ds{Z}_{>0}}$ is a field formed by union of field inclusions $\ds{Q}\subset\ds{Q}[2^{1/2}]\subset \ds{Q}[2^{1/2},2^{1/3}]\subset\ldots$.
The standard argument of Zorn's Lemma shows that every proper ideal of the ring $k[T_1,\ldots, T_n]_{\ds{Q}}$ is contained within a maximal ideal $\id{m}$.
\end{remark}

\subsection{Ideal}
 The affine algebraic set $V$ can be defined by different polynomials of $A$. There exists a natural way to assign an ideal in the ring $A$ to the set $V$.
\begin{define}
  Let $V$ be a subset of $k^n$; the ideal of $V$, denoted as $I(V)$, is the set of polynomials vanishing on $V$.
  \begin{equation}
    I(V)=\{f\in k[T_1,\ldots,T_n]_{\ds{Q}}\text{ such that }f(x)=0\text{ for all }x\in V\}.
  \end{equation}
  The regular functions on $V$ are defined as $k[T_1,\ldots,T_n]_{\ds{Q}}/I(V)$ also called the affine algebra of $V$ and is denoted as $k[V]_{\ds{Q}}$ or $\Gamma(V)$ (global sections).
\end{define}
Some properties are listed below, these are all analogous to the case of standard algebraic geometry.
\begin{enumerate}
  \item $V\subset V(I(V))$ and $V(I(V))$ is the closure of $V$ in the Zariski topology. If $V$ is affine algebraic set then $V(I(V))=V$.
  \item $S\subset I(V(S))$.
  \item $I(\varnothing)=A$
\end{enumerate}

\begin{proposition}
  For all irreducible affine algebraic sets, it is necessary and sufficient that its ideal is prime.
\end{proposition}
\begin{proof}
  The proof is word for word the same as in \cite[pp. 61]{reid1988undergraduate}
\end{proof}

The above proposition can be expressed by saying that $A/I(V)$ is integral.

\begin{corollary}
  If $k$ is infinite, then $k^n$ is irreducible.
\end{corollary}
\begin{proof}
Since, $k$ is infinite, the polynomials vanishing on all of $k^n$ is just zero. This gives $I(k^n)=(0)$ a prime ideal. The first sentence follows from the fact that every polynomial with rational degree can be converted to an integer degree polynomial by a change of variable, and the corresponding statement holds for integer degree polynomials as shown in  \cite[pp 12, Prop 2.4]{maclean2007algebraic}.
\end{proof}

\subsection{Nullstellensatz}
In this section the weak Nullstellensatz is proved by adapting \cite[p 17-18]{beltrametti2009lectures} and \cite[pp 15]{maclean2007algebraic}.
Let $k$ be an uncountable algebraically closed field and $A:=k[T_1,\ldots, T_n]_\ds{Q}$.
\begin{theorem}\label{nulstell} Let $k$ be algebraically closed and uncountable and  $I\subsetneq k[T_1,\ldots,T_n]_{\ds{Q}}$ be a proper ideal.
 Then $V(I)$ is non empty.
\end{theorem}

\begin{proof}
Using Zorn's Lemma the ideal $I\subsetneq A$ is contained within some maximal ideal $\id{m}$ of $A$. Let $B=A/\id{m}$ where $\id{m}$ is a maximal ideal and thus $B$ is a field. Since, $B$ is generated over $k$ by the monomials in $T_1,\ldots, T_n$ and their rational powers, the dimension $\dim_kB$ (as a $k$ vector space) is at most countable over $k$. Pick $b\in B\bs k$ and consider the family of elements
\begin{equation}
  \left\{\frac{1}{a-b}\right\},\quad a\in k \quad\text{ with $b$ fixed}.
\end{equation}
This family is uncountable because $k$ is uncountable. Since, $\dim_kB$ is countable, the elements of the family are linearly dependent.
Hence, there exists a relationship of the form
\begin{equation}
  \frac{\lambda_1}{a_1-b}+\ldots +\frac{\lambda_j}{a_j-b}=0\text{ for some }j\in\ds{Z}_{>0}
\end{equation}
Multiplying throughout by $\prod_i(a_i-b)$ and setting $a_i=X$ gives a polynomial $f(X)\in k[X]$ of degree $>0$ such that $f(b)=0$, showing that $b$ is algebraic over $k$. Since this holds for any arbitrary $b\in B$, the field $B$ is algebraic over $k$. But, $k$ is algebraically closed, hence $B=k$.

Consider the images of $T_1,\ldots,T_n$ in the field $B=k$, given as $a_1,\ldots,a_n$, in other words  \begin{equation}
  a_j^{1/i}:=T_j^{1/i}\mod\id{m}\qquad i\in\ds{Z}_{>0}.
\end{equation}
If $P(T_1,\ldots, T_n)\in\id{m}$ then $P(a_1,\ldots, a_n)=0$ or the point $(a_1,\ldots,a_n)\in k^n$ is in $V(I)$.
\end{proof}
The following corollary requires finite generation of the ideal $J$.
\begin{corollary}
 Let $J$ be a finitely generated ideal of $A$,then $I(V(J))=\rad(J)$.
\end{corollary}

\begin{proof}
The proof is via Rabinowitsch's trick as in \cite[pp. 25]{hulek2003elementary}.
Let $J$ be a finitely generated ideal of $A=k[T_1,\ldots,T_n]_{\ds{Q}}$ and $f\in I(V(J))$ an arbitrary element, it needs to be shown that $f^N\in J$ for some $N\in\ds{Z}_{>0}$.

The trick is to introduce a new variable $t$ and define
\begin{equation}
  J':=(J,ft-1)\triangleleft A[t],
\end{equation}
which implies that
\begin{equation}
  V(J')=\{(a_1,\ldots,a_n,b)\in k^{n+1}\text{ such that }(a_1,\ldots, a_n)\in V(J) \text{ and }bf(a_1,\ldots, a_n)=1\}
\end{equation}
Let $\pi$ be the projection of $V(J')$ to the first $n$ coordinates giving $\pi(V(J'))\subset V(J)$ such that $f(c_1,\ldots,c_n)\neq 0$ for $(c_1,\ldots,c_n)\in \pi(V(J'))$. Since, $f\in I(V(J))$ hence $V(J')=\varnothing$. Hence, $J'=A[t]$ by theorem \ref{nulstell}. Since $1\in A[t]=J'$, this implies
\begin{equation}
  1=\sum_{i=1}^r g_if_i+g_0\cdot(ft-1)\in A[t],
\end{equation}
where $\lr{f_1,\ldots, f_r}$ generate $J$ and $g_0,g_1,\ldots, g_r\in A[t]$.
The variable $t$ might appear in all the polynomials $g_0,\ldots, g_r$, let $N\in\ds{Z}_{>0}$ be the highest degree of $t$ in all the $g_i$. Multiplying the above equation with $f^N$ allows to set $(ft)^d=1$ for $d\leq N$ in the polynomials $g_i$ and giving us new polynomials $G_i$ in $A[f]=A$.
\begin{equation}
  \begin{aligned}
        f^N&=\sum_{i=1}^r f^Ng_if_i+f^Ng_0\cdot(ft-1),\\
        &=\sum_{i=1}^r f^Ng_if_i\mod (ft-1)\quad \text{set $ft=1$ to get}\\
          &\equiv \sum_{i=1}^r G_if_i\quad\text{ where } G_i\in A[f]=A
  \end{aligned}
\end{equation}
Since $f_i$ generate $J$ this gives us that $f^N\in J$ in the ring $A[t]/(ft-1)$.

Note that there is an injection $A\hookrightarrow A[t]/(ft-1)$ since $A$ is an integral domain and therefore injects into localization. Thus, the relation $f^N\in J$ also holds in $A$.
\end{proof}

\section{Noether Normalization}\label{normalization}
This section adapts the arguments from \cite[{Tag 00OW}]{stacks-project} to the case at hand.

\begin{define}

Let $I\subset k[T_1,\ldots,T_n]_{\ds{Q}}$ be an ideal, then $k[T_1,\ldots,T_n]_{\ds{Q}}/I$ will be called rationally generated over $k[T_1,\ldots,T_n]_{\ds{Q}}$.
\end{define}

\begin{remark}
  If the degrees are integer then there is finite generation, but this is no longer the case here. For example $k[X]_{\ds{Q}}/\lr{X}$ is infinitely generated by $\{1,X^r\}, r\in\ds{Q}\cap(0,1)$. This also shows that the monomials $(X^i)_{i\in{\ds{Q}_{>0}}}$ are linearly dependent for $r\in\ds{Q}\cap[1,\infty)$ in the ring $k[X]_{\ds{Q}}/\lr{X}$.
\end{remark}

\begin{lemma}\label{normalization-1}
  Let $g\in k[X_1,\ldots, X_n]$ be a non constant polynomial and $e_i\in\ds{Z}_{>0}$ such that $e_1\gg e_2\gg \ldots\gg e_{n-1}\gg e_n=1$, then the following holds with $d>0$ and $a\in k\bs\{0\}$
  \begin{equation}\label{eq-normalization-1}
    g(X_1+X_n^{e_1}, X_2+X_n^{e_2},\ldots, X_{n-1}+X_n^{e_{n-1}}, X_n)=aX_n^d+\text{ lower order terms in }X_n.
  \end{equation}
\end{lemma}

\begin{lemma}\label{normalization-1a}
  Let $f\in k[T_1,\ldots, T_n]_{\ds{Q}}$ be a non constant polynomial and $a_i\in\ds{Q}_{>0}$ such that $a_1\gg a_2\gg \ldots\gg a_{n-1}\gg a_n=1$, then the following holds with $d'>0$ and $a\in k\bs\{0\}$.
  \begin{equation}
    f(T_1^{1/j_1}+T_n^{a_1}, T_2^{1/j_2}+T_n^{a_2},\ldots, T_{n-1}^{1/j_{n-1}}+T_n^{a_{n-1}}, T_n)=aT_n^{d'}+\text{ lower order terms in }T_n.
  \end{equation}
  The $j_i\in\ds{Z}_{>0}$ for $1\leq i \leq n$.
\end{lemma}
\begin{proof}
  Convert the rational degree polynomial $f$ into an integer degree polynomial $g\in k[X_1,\ldots, X_n]$ as done in the proof of lemma \ref{lem-change-var}, by substitution $T_i^{1/j_i}\mapsto X_i$, then \eqref{eq-normalization-1} gives
  \begin{equation}
    \begin{aligned}
        g(X_1+X_n^{e_1}, X_2+X_n^{e_2},\ldots, X_{n-1}+X_n^{e_{n-1}}, X_n)&=aX_n^d+\text{ lower order terms in }X_n.\\
        f(T_1^{1/j_1}+T_n^{a_1}, T_2^{1/j_2}+T_n^{a_2},\ldots, T_{n-1}^{1/j_{n-1}}+T_n^{a_{n-1}}, T_n^{1/j_n})&=aT_n^{d'}+\text{ lower order terms in }T_n.
    \end{aligned}
  \end{equation}
Since $T_n^{1/j_n}\mapsto X_n$, this gives $a_i=e_i/j_n$ in the equation above.
\end{proof}

\begin{proposition}\label{prop-normalization-1}
  Let $k$ be a field and $A=k[T_1,\ldots,T_n]_{\ds{Q}}$. Let $S'=A/I$ where $I$ is a proper ideal of $A$. Then, there exist $Y_1^{1/j_1},\ldots, Y_{n-1}^{1/j_{n-1}}\in A$ such that $S'$ is rationally generated over $k[Y_1,\ldots, Y_{n-1}]_{\ds{Q}}$.
\end{proposition}
\begin{proof}
It suffices to prove for $S=A/f$ where $f$ is non zero element of $A$, since $S'$ is quotient ring of $S$. Set $Y_i^{1/j^i}=T_i^{1/j^i}-T_n^{a_i}$ for $i=1,\ldots,n$ for suitable rationals $a_i$. It suffices to show that $\bar{T}_n^{1/j_n}\in S$ satisfies a polynomial with rational degree with coefficients in $k[Y_1,\ldots,Y_{n-1}]_{\ds{Q}}$.The expression $f(T_1^{1/j_1},\ldots, T_n^{1/j_n})$, transforms to
  \begin{equation}
    f(Y_1^{1/j_1}+T_n^{a_1},\ldots+Y_{n-1}^{1/j_{n-1}}+T_n^{a_{n-1}},T_n^{1/j_n})=0
  \end{equation}
  for $\bar{T}_n^{1/j_n}\in S$ over $k[Y_1,\ldots,Y_{n-1}]_{\ds{Q}}$. The above equation gives via lemma \ref{normalization-1a}.
  \begin{equation}
    aT_n^{d'}+\text{ lower order terms in }T_n=0
  \end{equation}
  implying generation over $k[Y_1,\ldots, Y_{n-1}]_{\ds{Q}}$.
\end{proof}

\begin{lemma}
  Let $k$ be a field and $A=k[T_1,\ldots,T_n]_{\ds{Q}}$. Let $S'=A/I$ where $I$ is a proper ideal of $A$. Then there exists $r\geq 0$ and $Y_1,\ldots, Y_r\in A$ such that $S'$ is rationally generated over $k[Y_1,\ldots,Y_r]_{\ds{Q}}$.
\end{lemma}
\begin{proof}
  Proof by induction. For $n=0$ choose $r=0$ as well. For $n>0$ if $T_i$ are algebraically independent (that is $I=0$) then take $T_i=Y_i$. If $I\neq 0$ then choose $Y_1,\ldots, Y_{n-1}$ by Proposition \ref{prop-normalization-1}. Let $S\subset S'$ be the subring generated by image of rational powers of $Y_i$. By induction $r$ can be chosen such that the result holds, that is keep modding out by ideal $I$ till algebraically independent basis in terms of rational powers of $Y_i$ is obtained.
\end{proof}

\section{Rational functions and Morphisms}
\begin{define}
  Let $V\subset k^n$ and $W\subset k^m$ be two affine algebraic sets and let $\phi=(\phi_1,\ldots,\phi_m)$ with $\phi_i:V\ra k$ such that there is a map $\phi:V\ra W$.  Then $\phi$ is called regular if every $\phi_i\in k[V]_\ds{Q}$.
\end{define}

\textdbend The composition of two polynomials with rational degree might not be polynomial with a rational degree but a series associated with Newton's fractional Binomial expansion. For example let $f_1=1+X^{1/2}+X^2$ and $f_2=1+X$, then
\begin{equation}\label{prob-mor-1}
f_1\circ f_2= 1+(1+X)^{1/2}+(1+X)^2\notin k[X]_\ds{Q}
\end{equation}
On the other hand an evaluation at the point $x\in k$ leads to the equality
\begin{equation}
  (f_1\circ f_2 )(x)=f_1(f_2(x))\in k \text{ for all }x\in k \text{ algebraically closed}.
\end{equation}

Thus, it is not possible to obtain an antiequivalence of categories from morphism of varieties $V\ra W$ to a ring homomorphism $k[W]\ra k[V]$. This also shows that morphism of schemes is the correct approach to the problem.

\begin{define}
\begin{enumerate}
  \item Let $k(V)_\ds{Q}$ denote the field of fractions of $k[V]_\ds{Q}$, the elements $f\in k(V)_\ds{Q}$ are called rational functions on $V$.
  \item The rational function $f$ is regular at a point $p\in V$ if there is a representation of the form $f=g/h$ with $h(p)\neq 0$. The domain of definition of $f$ is given
  as
  \begin{equation}
    \mathrm{dom}(f):=\{p\in V\text{ such that }f\text{ is regular at }p\}.
  \end{equation}
  \item The local ring of $V$ at $p$ is given as
  \begin{equation}
    \mathscr{O}_{V,p}=\{f\in k(V)_\ds{Q}\text{ such that }f \text{ is regular at }p\}=k[V]_{\ds{Q}}[1/h]\text{ such that }h(p)\neq 0.
      \end{equation}
      The maximal ideal is $\id{m}_{V,p}=\{f\in  \mathscr{O}_{V,p}\text{ such that }f(p)=0\}$.
\end{enumerate}

\end{define}

\subsection{Finite Fields}\label{finite-field} In this section let $k$ be an algebraically closed field of $\chr p$, then $(X^{1/p}+Y^{1/p})^p=X+Y$ or $X^{1/p}+Y^{1/p}=(X+Y)^{1/p}$ in the ring $k[X,Y]_{\ds{Q}}$. This immediately leads us to consider the ring obtained from $k[X]$ by attaching $X^{1/p^i},i\in\ds{Z}_{>0}$.
\subsubsection{$\deg \ds{Z}[1/p]$ via Direct Limit}

\begin{construction}\label{constr-dir-lim-2-var}
  Consider the following inclusions
  \begin{equation}\label{dir-lim-2-var}
    k[X]\subseteq k[X,X^{1/p}]\subseteq k[X,X^{1/p},X^{1/p^2},X^{1/p^3}]\subseteq\ldots\subseteq\varinjlim_i k[X,\ldots, X^{1/p^i}]=\bigcup_i k[X,\ldots, X^{1/p^i}]
  \end{equation}
  The above construction can be carried onto multiple variables
  \begin{equation}
    \begin{aligned}
      k[X_1,\ldots, X_n]\subseteq k[X_1,X_1^{1/p},\ldots, X_n, X_n^{1/p}]\subseteq k[X_1,X_1^{1/p},X_1^{1/p^2},\ldots,X_n,X_n^{1/p},X_n^{1/p^2}]\subseteq\ldots\\
      \ldots\subseteq\varinjlim_i k[X_1,\ldots, X_1^{1/p^i},\ldots, X_n\ldots, X_n^{1/p^i}]=\bigcup_i k[X_1,\ldots, X_1^{1/p^i},\ldots,X_n,\ldots, X_n^{1/p^i} ]
    \end{aligned}
  \end{equation}
\end{construction}

\begin{define} The following notation will be used throughout
  \begin{equation}
    \begin{aligned}
      k[X]_{\ds{Z}[1/p]}&=\bigcup_i k[X,\ldots, X^{1/p^i}]\\
      k[X_1,\ldots, X_n]_{\ds{Z}[1/p]}&=\bigcup_i k[X_1,\ldots, X_1^{1/p^i},\ldots,X_n,\ldots, X_n^{1/p^i} ]
    \end{aligned}
  \end{equation}
\end{define}

Since, $(X_1+\ldots+X_n)^{1/p}=X_1^{1/p}+\ldots+X_n^{1/p}$ the composition of two polynomials with a rational degree is a polynomial and the \eqref{prob-mor-1} can now be expressed as
\begin{equation}
  f_1\circ f_2= 1+(1+X)^{1/2}+(1+X)^2= 1+1+X^{1/2}+(1+X)^2\in k[X]_{\ds{Z}[1/2]}\text{ where }\chr k=2.
\end{equation}
In what follows,
\cite{reid1988undergraduate} and \cite{hulek2003elementary} are adapted to the case at hand.
\begin{define}
Let $V\subset \ds{A}^n_k$ and $W\subset\ds{A}^n_k$ be algebraic sets, a map $\phi:V\ra W$ is a polynomial map if there are $\phi_1,\ldots,\phi_m\in k[X_1,\ldots, X_n]_{\ds{Z}[1/p]}$ such that
\begin{equation}
\phi(p)=(\phi_1(p),\ldots,\phi_m(p))\text{ for all }p\in V
\end{equation}
\end{define}

Since composition of two polynomials with degree $\ds{Z}[1/p]$ over a field of $\chr p$ is again an element of $k[X_1,\ldots, X_n]_{\ds{Z}[1/p]}$, the maps between algebraic sets can be composed. If $V\subset\ds{A}^{n}, W\subset\ds{A}^{m}$ and $U\subset \ds{A}^{\ell}$ are algebraic sets, and $\phi:V\ra W,\varphi:W\ra U$ are polynomial maps then
$\varphi\circ \phi:V\ra U$ is again a polynomial map. If $\phi$ is given by $(\phi_1,\ldots,\phi_{m} )$ and $\varphi$ by $(\varphi_1,\ldots,\varphi_{\ell})$, then $\varphi\circ\phi$ is given by

\begin{equation}
  \varphi_1(\phi_1,\ldots,\phi_{m} ),\ldots,\varphi_{\ell}(\phi_1,\ldots,\phi_{m} )\in k[X_1,\ldots, X_n]_{\ds{Z}[1/p]}
\end{equation}
The following is an analogue of \cite[pp 73]{reid1988undergraduate}
\begin{theorem}\label{finite-field-thm1}
  Let $V\subset \ds{A}^n,W\subset A^m$ and $U\subset A^\ell$ be algebraic sets as above.
  \begin{enumerate}
     \item A polynomial map $\phi: V \ra W $ induces a $k$-algebra homomorphism $\phi^* : k[W]_{\ds{Z}[1/p]} \ra k[V]_{\ds{Z}[1/p]}$, defined by
composition of functions; that is, if $g \in k[W]_{\ds{Z}[1/p]}$ is a polynomial function then so is $\phi^* (g) = g \circ \phi$.
\item  If $\phi:V\ra W$ and $\varphi:W\ra U$ are polynomial maps then $(\varphi\circ\phi)^*=\phi^*\circ\varphi^*$.
\item If $\Phi:k[W]\ra k[V]$  then it is of the form $\Phi=\phi^*$ where $\phi:V\ra W$ is a unique polynomial map.
  \end{enumerate}
\end{theorem}
\begin{proof}
  \begin{enumerate}
    \item For $g \in k[W]_{\ds{Z}[1/p]}$ define $\phi^*(g)=g\circ \phi=g(\phi_1,\ldots, \phi_m)\in k[V]_{\ds{Z}[1/p]}$. The map $\phi^*$ is a ring homomorphism since
    \begin{equation}
      \begin{aligned}
        \phi^*(g_1+g_2)&=(g_1+g_2)\circ \phi=g_1\circ\phi+g_2\circ\phi=\phi^*(g_1)+\phi^*(g_2)\\
        \phi^*(g_1\cdot g_2)&=(g_1\cdot g_2)\circ \phi=(g_1\circ\phi)\cdot (g_2\circ\phi)=\phi^*(g_1)\cdot\phi^*(g_2)
      \end{aligned}
        \end{equation}
    \item Let $h\in k[U]_{\ds{Z}[1/p]}$
    \begin{equation}
      (\varphi\circ\phi)^*(h)=h\circ \varphi \circ \phi=\varphi^*(h)\circ \phi=\phi^*\circ\varphi^*(h).
    \end{equation}
\item Since, $W\subset\ds{A}_k^m$ this gives
\begin{equation}
  k[W]=\frac{k[Y_1,\ldots, Y_m]_{\ds{Z}[1/p]}}{I(W)}=k[y_1,\ldots, y_m]_{\ds{Z}[1/p]}\text{ where }y_i^{1/p^j}=Y_i^{1/p^j}+I(W)
\end{equation}
for $j\in\ds{Z}[1/p]_{\geq 0}$. Since, $\Phi:k[W]_{\ds{Z}[1/p]}\ra k[V]_{\ds{Z}[1/p]}$ is given define $\phi_i\in k[V]_{\ds{Z}[1/p]}$ as $\phi_i=\Phi(y_i)$ This gives a polynomial map $\phi:=(\phi_1,\ldots,\phi_m): V\ra \ds{A}^m$. First, it is shown that $f(V)\subset W$. Let $G\in I(W)\subset k[Y_1,\ldots, Y_m]_{\ds{Z}[1/p]}$; then
\begin{equation}
G(y_1,\ldots, y_n)=0\in k[W],
\end{equation}
which is simply substitution of $Y_i$ with $y_i$. This gives $\Phi(G(y_1,\ldots,y_m))=0\in k[V]$. Since $\Phi$ is a $k$ algebra homomorphism
\begin{equation}
  0=\Phi(G(y_1,\ldots, y_m))=G(\Phi(y_1),\ldots,\Phi(y_m))=G(\phi_1,\ldots, \phi_m)\Rightarrow f(V)\subset W.
\end{equation} It needs to be checked that $\phi^*=\Phi$ and it is enough to check on the generators of $k[W]_{\ds{Z}[1/p]}$, that is $\Phi(y_i)=\phi^*(y_i)=\phi_i$. But, this holds since it is the definition of $\phi_i$. This also shows that $\phi=(\phi_1,\ldots,\phi_m)$ is the unique polynomial with $\Phi=\phi^*$.
  \end{enumerate}
\end{proof}

\section{Projective Geometry}\label{proj-geo}
A definition analogous to that of affine varieties is needed, but the problem is that the homogeneous coordinates $(x_0:\cdots:x_n)$ not being uniquely defined, even the value of polynomial at the point is not unique. In fact, only the zeros of the polynomials are of interest us. These will be defined for a certain
category of polynomials.

\begin{define}
  A rational degree polynomial $f\in k[T_1,\ldots, T_n]_{\ds{Q}}$ is homogeneous of degree $d\in\ds{Q}_{\geq 0}$ if
  \begin{equation}
    f=\sum_{i_0\ldots i_n} {a_{i_0\ldots i_n}} T_0^{i_0}\cdots T_n^{i_n}\text{ with }a_{i_0\ldots i_n}\neq 0 \text{ only if }i_0+\ldots+i_n=d.
  \end{equation}
  Any $f\in k[T_1,\ldots, T_n]_{\ds{Q}}$ has a unique expression $f=f_0+\ldots+ f_N$ with terms arranged with increasing degree in which each $f_d$ is homogeneous of degree $d\in\ds{Q}_{\geq 0}$.
\end{define}
\begin{proposition}
  If $f$ is homogeneous of degree $d$ then
  \begin{equation}
    f(\lambda T_0,\ldots, \lambda T_n)=\lambda^d f(T_0,\ldots, T_n) \text{ for all }\lambda \in k;
  \end{equation}
  if $k$ is infinite then converse holds too.
\end{proposition}

\begin{proof}
  Only the converse needs to be proved. If $k$ is infinite then $f(X)$ vanishes on a finite set of $k$. Thus, $  f(\lambda T_0,\ldots, \lambda T_n)=\lambda^d f(T_0,\ldots, T_n)$ for almost all $k$ implying homogeneous of degree $d$.
\end{proof}

An immediate consequence is that if $F$ is homogeneous for $\lambda\in k\bs\{0\}$ \[F(x_1,\ldots,x_0)=0\iff F(\lambda x_0,\ldots, \lambda x_n)=0 \]

\subsection{Graded Rings and Homogeneous Ideals}
Recall that a ring or a module can be graded over a commutative monoid $\Delta$ as shown in \cite[p 363, Chapter II, \S 1]{bourbaki1998algebra}.

The graded polynomial ring is given as
\begin{equation}
  \begin{aligned}
    A&=\bigoplus_{d\in\ds{Q}_{\geq 0}} A_d,\quad\text{ where }\\
    A_d&=\{f\in k[T_0,\ldots,T_n]_{\ds{Q}}\text{ such that }f \text{ is homogeneous of degree }d\}\cup\{0\}
  \end{aligned}
   \end{equation}

\begin{define}
  An ideal $I\triangleleft k[T_0,\ldots T_n]_{\ds{Q}}$ is homogeneous if for all $f\in I$, there is a homogeneous decomposition of $f$
  \begin{equation}
    f=f_0+\ldots+f_N,\quad f_i\in I\text{ is homogeneous of degree $d_i$ for all }i.
  \end{equation}
The homogeneous ideal satisfies $I=\bigoplus_{d\geq 0}(I\cap A_d)$.
\end{define}
\begin{define}
  The irrelevant ideal is given by
  \begin{equation}
    A_+=\bigoplus_{d>0}A_d.
  \end{equation}
  The above can also be expressed as $A_+=\bigcup_{i\in\ds{Z}_{>0}}(T_0,\ldots, T_0^{1/i},\ldots, T_n,\ldots, T_n^{1/i})$.
\end{define}

\begin{define}
  \begin{enumerate}
    \item Let $S\subset k[T_0,\ldots,T_n]_{\ds{Q}}$ be a set of homogeneous polynomials, the projective algebraic set defined by $S$ is given as
    \begin{equation}
      V_p(S)=\{x\in\ds{P}^n_k\text{ such that }F\in S\text{ for all } F(x)=0 \}.
    \end{equation}
    \item Let $V\subset \ds{P}^n_k$, then the ideal of $V$ is given as
    \begin{equation}
      I_p(V)=\{F\in k[T_0,\ldots, T_n]_{\ds{Q}}\text{ such that }F(x)=0\text{ for all }x\in V\}.
    \end{equation}
    This is a homogeneous ideal.
    \item The graded ring associated to projective algebraic set $V$ is given as the graded quotient ring
    \begin{equation}
      \Gamma_h(V)=k[T_0,\ldots,T_n]_{\ds{Q}}/I_p(V)
    \end{equation}
    \item Let $f\in\Gamma_h(V)$ such that $\deg f>0$, the open sets are given as
    \begin{equation}
      D_+(f)=\{x\in V\text{ such that }f(x)\neq 0\}
    \end{equation}
  \end{enumerate}

\end{define}

Since, the rings at hand are non-noetherian it cannot be assumed that the set $S$ is finite. Once again, $V(I(V))$ is closure of $V$.
\begin{example}
The hyperplane at infinity is given by
\begin{equation}
  H_i=\{(x_0,x_1,\cdots:x_n)\in\ds{P}_k^n\text{ such that }x_i=0\}.
\end{equation}
Since, $k$ is a field this also translates to $x^j_i=0$ for $j\in\ds{Q}_{>0}$. The corresponding open set is given as
\begin{equation}
      U_i=\{(x_0,x_1,\cdots,x_n)\in\ds{P}_k^n\text{ such that }x_i\neq 0\}.
\end{equation}
\end{example}
The above example gives a covering of $\ds{P}_k^n=U_0\cup\ldots\cup U_n$. There is a bijection
\begin{equation}
\begin{aligned}
  j_i:U_i&\ra \ds{A}_k^n\\
  (x_0,\ldots,x_i,\ldots:x_n)&\mapsto \left(\frac{x_0}{x_i},\ldots, \frac{x_n}{x_i}\right)
\end{aligned}
\end{equation}

\section{Projective Nullstellensatz}
\subsection{Affine cone} Let $\pi:\ds{A}_k^{n+1}\{0,\ldots,0\}\ra\ds{P}^n_k$ be the canonical projection and $I\subset k[T_0,\ldots, T_n]_{\ds{Q}} $ a homogeneous ideal. There are two sets corresponding to this ideal the projective zero set $V(I)\in\ds{P}^n_k$ and the affine zero set $V^a(I)\subset\ds{A}^n_k $, also called the affine cone over projective set $Y=V(I)$ and denoted as $C(Y)$.
\begin{equation}
  \begin{aligned}
    V^a(I)=\pi^{-1}(V(I))\cup\{0,\ldots, 0\} \text{ for }I\subsetneq k[T_0,\ldots, T_n]_{\ds{Q}}.\\
    (x_0,\ldots,x_n)\in V^a(I)\iff (\lambda x_0,\ldots,\lambda x_n)\in  V^a(I) \text{ for all }\lambda\in\tio{k}.
  \end{aligned}
  \end{equation}
If $I=k[T_0,\ldots, T_n]_{\ds{Q}}$ then $C(Y)=0$. The affine cone often helps to reduce a projective problem to an affine problem.

\begin{theorem}
  Let $k$ be an algebraically closed field and $J\subset k[T_0,\ldots, T_{n+1}]_\ds{Q}$ a finitely generated homogeneous ideal, then:
  \begin{enumerate}
    \item $V(J)=\varnothing$ if and only if $\rad(J)\supset (T_0,\ldots,T_n)$.
    \item If $V(J)\neq\varnothing$, then $I(V(J))=\rad(J)$.
  \end{enumerate}
\end{theorem}
\begin{proof}
  \begin{enumerate}
    \item   Let $x\sim y$ iff $x=\lambda y, \lambda\in\tio{k}$ for $x,y\in V^a(J)$, then
      \begin{equation}
        \begin{aligned}
          \ds{P}^n_k\supset V(J)&= (V^a(J)-\{0,\ldots, 0\})/\sim\\
          \text{ or }V(J)&=\varnothing \iff V^a(J)\subset\{0,\ldots,0\}
        \end{aligned}
      \end{equation}
    Affine Nullstellensatz then gives $\rad(J)\supset (T_0,\ldots, T_n)$.

    \item If $V(J)\neq \varnothing$, then from affine Nullstellensatz
    \begin{equation}
      f\in I(V(J))\iff f\in I(V^a(J)) \iff f\in\rad(J).
    \end{equation}
  \end{enumerate}

\end{proof}
\subsection{Standard Affine Charts} Let $X\subset\ds{P}_k^n$ be an algebraic set not contained in any of the hyperplanes of $\ds{P}_k^n$, then $X$ can be covered with $X_{(i)}$ where each $X_{(i)}=X\cap U_i$. Hence, each $X_{(i)}\subset \ds{A}_k^n$ is an affine algebraic set. If $I(X)$ is the homogeneous ideal associated to $X$ and $X_{(n)}=X\cap U_n$, then the following hold
\begin{equation}
  \begin{aligned}
    I(X_{(n)})&=\{f(T_0,\ldots, T_{n-1},1)\text{ such that }f\in I(X)\}\\
    I(X)_d&=\left\{T^{d}_n\cdot f\left(\frac{T_0}{T_{n}},\ldots, \frac{T_{n-1}}{T_{n}}\right)\text{ such that }f\in I(X_{(n)})\text{ and }\deg f\leq d\right\}
  \end{aligned}
\end{equation}
where $I(X)_d$ denotes the degree $d$ homogeneous part. Hence, there is a correspondence
\begin{equation}
  \{\text{algebraic sets $X\subset \ds{P}_k^n$ such that $X\not\subset H_i$}\}\leftrightarrow   \{\text{algebraic sets $X_{(i)}\subset U_i\simeq \ds{A}_k^n$} \}
\end{equation}

\section{Schemes}\label{sch1} Let $X=\spc~A$ endowed with Zariski topology and $f\in A$, the open set and closed sets are given as
\begin{equation}
  D(f):=\spc~A\bs V(f)\qquad V(f):=\{\id{p}\in\spc~A\text{ such that }f\in\id{p}\}
\end{equation}
Recall that $\spc~A$ is quasicompact and $V(I)\subset V(J)$ if and only if $J\subset\rad{I}$. Thus it is possible to cover the space with finitely open sets $\{D(f_i)\}$ which form the base of the space and define the sheaf on open sets as $\mathscr{O}_X(D(f))=A_f$. The localization is done at the multiplicatively closed set $\{1,f,f^2,f^3,\ldots\}$. Hence, the affine scheme is of the form
\begin{equation}
  (\spc~k[T_1,\ldots,T_n]_{\ds{Q}},\mathscr{O}_{\spc~k[T_1,\ldots,T_n]_{\ds{Q}}}).
  \end{equation}
The affine pieces can be glued together to get projective space. Let $0\leq i,j\leq n$ and
\begin{equation}
  X_i=\spc~k\left[\frac{T_0}{T_i},\ldots,\frac{T_n}{T_i}\right]_{\ds{Q}},\qquad X_{ij}=D\left(\frac{T_j}{T_i}\right)\subset X_i.
\end{equation}

Note that $X_{ii}=D(1)$ and thus corresponds to all of $X_i$ (none of the prime ideals of $X_i$ contain 1). This gives
\begin{equation}
  \begin{aligned}
    \mathscr{O}_{X_i}(X_{ij})&=k\left[\frac{T_0}{T_i},\ldots,\frac{T_j}{T_i},\ldots\frac{T_n}{T_i}\right]_{\ds{Q}}\text{ localised at }\frac{T_j}{T_i}\\
    \mathscr{O}_{X_i}(X_{ij})&=k\left[\frac{T_0}{T_i},\ldots,\frac{T_j}{T_i},\ldots\frac{T_n}{T_i},\frac{T_i}{T_j} \right]_{\ds{Q}}\\
  \text{Change Variable to get an} & \text{  isomorphism}\\
    &=k\left[\frac{T_0}{T_j},\ldots,\frac{T_i}{T_j},\ldots\frac{T_n}{T_j},\frac{T_j}{T_i} \right]_{\ds{Q}}\\
    &=k\left[\frac{T_0}{T_j},\ldots,\frac{T_i}{T_j},\ldots\frac{T_n}{T_j}\right]_{\ds{Q}}\text{ localised at }\frac{T_i}{T_j}\\
    &=\mathscr{O}_{X_j}(X_{ji})
  \end{aligned}
\end{equation}
\begin{remark}
  The ring $k[X,Y]_{\ds{Q}}$ can be localized at the multiplicatively closed set $\{X^r\}_{r\in\ds{Q}_{\geq0}}$ or at the multiplicatively closed set $\{1,X,X^2,\ldots\}=\{X^r\}_{r\in\ds{Z}_{\geq 0}}$, it will yield the same ring. For some $f\in k[X,Y]_{\ds{Q}}$, the set
  $f^r,r\in\ds{Q}_{>0}$ might not be defined, or $f^r$ may yield a power series via fractional newton Binomial theorem which does not lie in the ring  $k[X,Y]_{\ds{Q}}$. The main idea is to work with multiplicatively closed set with integer degrees, so that the results from standard algebraic geometry are carried over verbatim.
\end{remark}

\section{Proj and Twisting Sheaves $\mathscr{O}(n)$}\label{sec-twist}

Let $A$ be a graded ring over $\ds{Q}$ it can be endowed with decomposition $A=\oplus_{d\geq 0}A_d$ of abelian groups such that $A_dA_e\subseteq A_{d+e}$ for all $d,e\geq 0$ where $d,e\in\ds{Q}$. Similarly graded $A$ modules can be defined with $A_dM_e\subseteq M_{d+e}$ for all $d,e\geq 0$ and $d,e\in\ds{Q}.$ A homogeneous ideal is of the form $I=\oplus_{d\geq 0}(I\cap A_d)$ and the quotient $A/I$ has a natural grading $(A/I)_d=A_d/(I\cap A_d)$. Let $\Proj A$ denote the set of prime ideals of $A$ that do not contain $A_+:=\oplus_{d>0}A_d$, then $\Proj A$ can be endowed with the structure of a scheme. The closed and open sets for homogeneous ideals $I$ are of the form
\begin{equation}
  \begin{aligned}
    V_+(I)&:=\{\id{p}\in\Proj A\text{ such that }I\subseteq \id{p}\}\\
    D_+(f)&:=\Proj A\bs V(fA).
  \end{aligned}
\end{equation}

Let $n\in\ds{Q}$ and $A$  a $\ds{Q}$ module, define a new graded $A$ module  $A(n)_d:=A_{n+d}$ for all $d\in \ds{Q}$, define $\mathscr{O}_X(n):=A(n)\tilde{\phantom{b}}$. The $\tilde{\phantom{b}}$ denotes the localization on the affine open. Thus, on $D_+(f)$ an affine open subset  $\mathscr{O}_X(n)\vert_{D_+(f)}=f^n\mathscr{O}_X\vert_{D_+(f)}$, furthermore, the usual equality holds $\mathscr{O}_X(n)\otimes_{\mathscr{O}_X}{\mathscr{O}_X(m)}=\mathscr{O}_X(n+m)$.

\subsection{$\mathscr{O}(1)$} Consider $\Proj~k[X_0,X_1]_{\ds{Q}}$, the affine open sets are
\begin{equation}
  U_0=D(X_0)=\spc~k[(X_1/X_0)]_{\ds{Q}}\qquad \text{ and }\qquad U_1=D(X_1)=\spc~k[(X_0/X_1)]_{\ds{Q}}.
\end{equation}
The transition function from $U_1$ to $U_0$ is given as $X_1/X_0$. Let $r\in \ds{Q}\cap (0,1)$

\begin{equation}
  \begin{aligned}
\text{Global Section}&\quad    X_0+X_0^{r}X_1^{1-r}+X_1\\
  U_1  &\quad(X_0/X_1)+(X_0/X_1)^{r}+1\\
  U_0  &\quad 1+(X_1/X_0)^{1-r}+(X_1/X_0)
  \end{aligned}
\end{equation}

\subsection{Computing Cohomology}\label{Computing-Coho}
The global sections of degree 2 of $\Proj~R[X,Y]$ are generated by $X^2,XY,Y^2$, where as the global sections of degree 2 of $\Proj~R[X,Y]_{\ds{Q}}$ are given as $X^2, Y^2, XY, X^{r_i}Y^{1-r_i}$ for $r_i\in\ds{Q}\cap(0,2)$, and are thus infinitely many.

\begin{theorem}\label{t1}
  Let $S=R[X_0,\ldots, X_n]_{\ds{Q}}$ and $X=\Proj~S$, then for any $n\in\ds{Q}$
\begin{enumerate}
\item There is an isomorphism $S\simeq \oplus_{n\in\ds{Q}} H^0(X,\mathscr{O}_X(n))$.
\item $H^n(X,\curly{O}_{X}(-n-1))$ is a free module of infinite rank.
\end{enumerate}
\end{theorem}

\begin{proof}
\begin{enumerate}
  \item
  Take the standard cover by affine sets $\id{U}=\{U_i\}_{i}$ where each $U_i=D(X_i), i=0,\ldots,n$. The global sections are given
  as the kernel of the following map
   \begin{equation}
   \prod S_{X_{i_0}}\longrightarrow \prod S_{X_{i_0}X_{i_1}}
    \end{equation}

    The element mapping to the Kernel has to lie in all the intersections $S=\cap_i S_{X_i}$, as given on \cite[pp 118]{hartshorne1977algebraic} and is thus the ring $S$ itself.

    \item  $H^n(X,\mathscr{O}_X(-m))$ is the cokernel of the map

    \begin{equation}
    d^{n-1}:\prod_k S_{X_0\cdots \hat{X}_k\cdots X_n}\longrightarrow S_{X_0\cdots X_n}
    \end{equation}

    $S_{X_0\cdots X_n}$ is a free $R$ module with basis $X_0^{l_0}\cdots X_n^{l_n}$ with each $l_i\in\ds{Q}$. The image of $d^{n-1}$ is the free submodule generated by those basis elements with atleast one $l_i\geq 0$. Thus $H^n$ is the free module with basis as negative monomials

    \begin{equation}
    \{X_0^{l_0}\cdots X_n^{l_n}\}\text{ such that }l_i<0
    \end{equation}

    The grading is given by $\sum l_i$ and there are infinitely many monomials with degree $-n-\epsilon$ where $\epsilon$ is something very small and $\epsilon\in\ds{Q}$. Recall, that in the standard coherent cohomology there is only one such monomial $X_0^{-1}\cdots X_n^{-1}$ . For example, in case of $\ds{P}^2$ we have $X_0^{-1}X_1^{-1}X_2^{-1}$ but for the case at hand we also have   $X_0^{-1/2}X_1^{-1/2}X_2^{-2}$.

    Recall that in coherent cohomology of $\ds{P}^n$ the dual basis of $X_0^{m_0}\cdots X_n^{m_n}$ is given by $X_0^{-m_0-1}\cdots X_n^{-m_n-1}$ and the operation of multiplication gives pairing. We do not have this pairing here, but we can pair $X_0^{m_0}$ with $X_0^{-m_0}$.
\end{enumerate}
  \end{proof}

\pagebreak

  \begin{theorem}\label{t2} Let $S=R[X_0,\ldots, X_n]_{\ds{Q}}$ and $X=\Proj~R$, then
$H^i(X,\curly{O}_X(m))=0$ if $0<i<n$.
\end{theorem}

\begin{proof}
  The proof from \cite[pp 474-475]{ravi2} is adapted to the case at hand, using the convention that $\Gamma$ denotes global sections. We will work with $n=2$ for the sake of clarity, the case for general $n$ is identical. The \v{C}ech complex is given in figure \ref{check1}.

  \begin{figure}[H]
  \centering
   \begin{tikzpicture}
   []
          \matrix (m) [
              matrix of math nodes,
              row sep=0.5em,
              column sep=2.5em,
                     ]
                     {
                      ~ & ~ &  |[name=qa1]|U_0 & |[name=qb1]|U_{0}\cap U_1 &~&~\\
                      ~ & ~ & \oplus & \oplus &~&~\\
                      |[name=qka]| 0 &|[name=qkb]| \ds{P}^2 & |[name=qkc]| U_1 &|[name=qkd]| U_{0}\cap U_2 & |[name=qke]| U_{0}\cap U_1\cap U_2 &|[name=qkf]|0\\
                       ~ & ~ & \oplus & \oplus &~&~\\
                      ~ & ~ & |[name=qa2]| U_2 & |[name=qb2]|U_1\cap U_2 &~&~\\
                      ~ & ~ &  |[name=a1]|\Gamma_{X_0}& |[name=b1]|\Gamma_{X_0 ,X_1  }&~&~\\
                      ~ & ~ & \oplus & \oplus &~&~\\
                      |[name=ka]| 0 &|[name=kb]| \Gamma & |[name=kc]|\Gamma_{X_1 }&|[name=kd]|\Gamma_{X_0 ,X_2  }&|[name=ke]|\Gamma_{X_0 ,X_1 ,X_2 }&|[name=kf]|0\\
                       ~ & ~ & \oplus & \oplus &~&~\\
                      ~ & ~ & |[name=a2]| \Gamma_{X_2}& |[name=b2]|\Gamma_{X_1 ,X_2  }&~&~\\
                             };
          \path[overlay,->, font=\scriptsize,>=latex]
          (ka) edge (kb)
           (kb) edge (kc)

            (kd) edge (ke)
            (ke) edge (kf)

              (kb) edge (a1)
              (kb) edge (a2)
               (a1) edge (b1)
   (a1) edge (kd)
   (a2) edge (b2)
  (a2) edge (kd)
  (kc) edge (b1)
  (kc) edge (b2)
  (b1) edge (ke)
  (b2) edge (ke)
   ;

    \path[overlay,<-, font=\scriptsize,>=latex]

          (qka) edge (qkb)
           (qkb) edge (qkc)
                      (qkd) edge (qke)
            (qke) edge (qkf)

              (qkb) edge (qa1)
              (qkb) edge (qa2)
               (qa1) edge (qb1)
   (qa1) edge (qkd)
   (qa2) edge (qb2)
  (qa2) edge (qkd)
  (qkc) edge (qb1)
  (qkc) edge (qb2)
  (qb1) edge (qke)
  (qb2) edge (qke)
   ;
  \end{tikzpicture}
  \caption{\v{C}ech Complex for $n=2$}\label{check1}
  \end{figure}
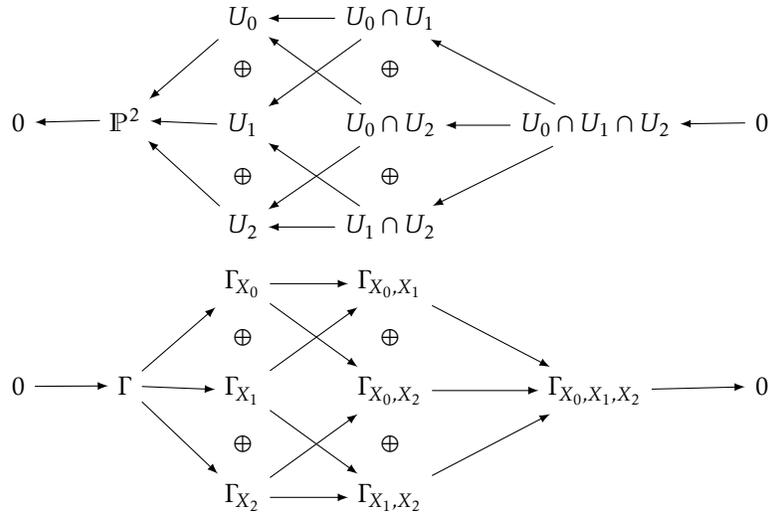

  \begin{description}
  \item[$3$ negative exponents] The monomial $X_0^{a_0}\cdot X_1^{a_1}\cdot X_2^{a_2}$ where $a_i<0$. We cannot lift it to any of the coboundaries (that is lift only to $0$ coefficients). If $K_{012}$ denotes the coefficient of the monomial in the complex (Figure \ref{check2}), we get zero cohomology except for the spot corresponding to $U_0\cap U_1\cap U_2$.

  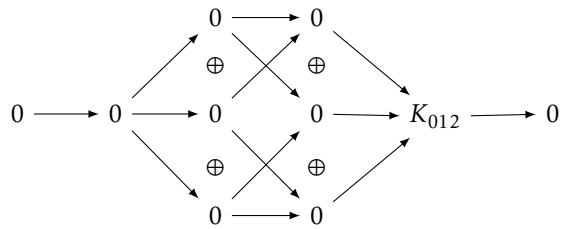
\begin{figure}[H]
  \centering
   \begin{tikzpicture}
   []
          \matrix (m) [
              matrix of math nodes,
              row sep=0.5em,
              column sep=2.5em,
                     ]
  { ~ & ~ &  |[name=a1]|0 & |[name=b1]|0 &~&~\\
   ~ & ~ & \oplus & \oplus &~&~\\
   |[name=ka]| 0 &|[name=kb]| 0 & |[name=kc]|0 &|[name=kd]|0 &|[name=ke]|K_{012}&|[name=kf]|0\\
    ~ & ~ & \oplus & \oplus &~&~\\
   ~ & ~ & |[name=a2]| 0& |[name=b2]|0 &~&~\\
          };

          \path[overlay,->, font=\scriptsize,>=latex]
          (ka) edge (kb)
           (kb) edge (kc)

            (kd) edge (ke)
            (ke) edge (kf)

              (kb) edge (a1)
              (kb) edge (a2)
               (a1) edge (b1)
   (a1) edge (kd)
   (a2) edge (b2)
  (a2) edge (kd)
  (kc) edge (b1)
  (kc) edge (b2)
  (b1) edge (ke)
  (b2) edge (ke)
   ;
  \end{tikzpicture}
  \caption{$3$ negative exponents}\label{check2}
  \end{figure}
  \pagebreak
  \item[$2$ negative exponents] The monomial $X_0^{a_0}\cdot X_1^{a_1}\cdot X_2^{a_2}$ where two exponents are negative, say $a_0,a_1<0$. Then we can perfectly lift to coboundary coming from $U_0\cap U_1$, which gives exactness.

  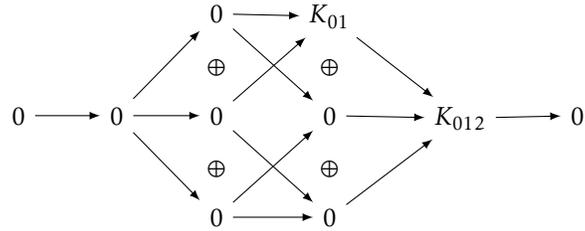
\begin{figure}[H]
  \centering
   \begin{tikzpicture}
   []
          \matrix (m) [
              matrix of math nodes,
              row sep=0.5em,
              column sep=2.5em,
                     ]
  { ~ & ~ &  |[name=a1]|0 & |[name=b1]|K_{01} &~&~\\
   ~ & ~ & \oplus & \oplus &~&~\\
   |[name=ka]| 0 &|[name=kb]| 0 & |[name=kc]|0 &|[name=kd]|0 &|[name=ke]|K_{012}&|[name=kf]|0\\
    ~ & ~ & \oplus & \oplus &~&~\\
   ~ & ~ & |[name=a2]| 0& |[name=b2]|0 &~&~\\
          };

          \path[overlay,->, font=\scriptsize,>=latex]
          (ka) edge (kb)
           (kb) edge (kc)

            (kd) edge (ke)
            (ke) edge (kf)

              (kb) edge (a1)
              (kb) edge (a2)
               (a1) edge (b1)
   (a1) edge (kd)
   (a2) edge (b2)
  (a2) edge (kd)
  (kc) edge (b1)
  (kc) edge (b2)
  (b1) edge (ke)
  (b2) edge (ke)
   ;
  \end{tikzpicture}
  \caption{$2$ negative exponents}\label{check3}
  \end{figure}

  \item[$1$ negative exponent] The monomial $X_0^{a_0}\cdot X_1^{a_1}\cdot X_2^{a_2}$ where one exponents is negative, say $a_0<0$, we get the complex (Figure \ref{check4}). Notice that $K_0$ maps injectively giving zero cohomology group.

  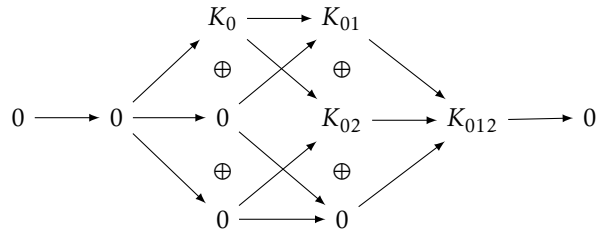
\begin{figure}[H]
  \centering
   \begin{tikzpicture}
   []
          \matrix (m) [
              matrix of math nodes,
              row sep=0.5em,
              column sep=2.5em,
                     ]
  { ~ & ~ &  |[name=a1]|K_0 & |[name=b1]|K_{01} &~&~\\
   ~ & ~ & \oplus & \oplus &~&~\\
   |[name=ka]| 0 &|[name=kb]| 0 & |[name=kc]|0 &|[name=kd]|K_{02} &|[name=ke]|K_{012}&|[name=kf]|0\\
    ~ & ~ & \oplus & \oplus &~&~\\
   ~ & ~ & |[name=a2]| 0& |[name=b2]|0 &~&~\\
          };

          \path[overlay,->, font=\scriptsize,>=latex]
          (ka) edge (kb)
           (kb) edge (kc)

            (kd) edge (ke)
            (ke) edge (kf)

              (kb) edge (a1)
              (kb) edge (a2)
               (a1) edge (b1)
   (a1) edge (kd)
   (a2) edge (b2)
  (a2) edge (kd)
  (kc) edge (b1)
  (kc) edge (b2)
  (b1) edge (ke)
  (b2) edge (ke)
   ;
  \end{tikzpicture}
  \caption{$1$ negative exponent}\label{check4}
  \end{figure}

  Furthermore, the mapping in the Figure \ref{check5} gives Kernel when $f=g$ which is possible for zero only. Again giving us zero cohomology groups.
  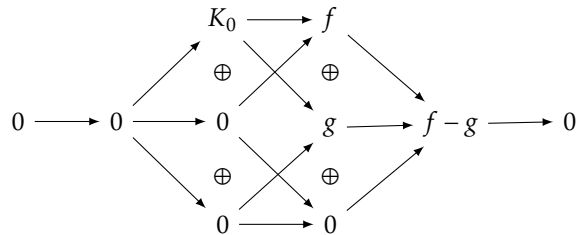
\begin{figure}[H]
  \centering
   \begin{tikzpicture}
   []
          \matrix (m) [
              matrix of math nodes,
              row sep=0.5em,
              column sep=2.5em,
                     ]
  { ~ & ~ &  |[name=a1]|K_0 & |[name=b1]|f &~&~\\
   ~ & ~ & \oplus & \oplus &~&~\\
   |[name=ka]| 0 &|[name=kb]| 0 & |[name=kc]|0 &|[name=kd]|g &|[name=ke]|f-g&|[name=kf]|0\\
    ~ & ~ & \oplus & \oplus &~&~\\
   ~ & ~ & |[name=a2]| 0& |[name=b2]|0 &~&~\\
          };

          \path[overlay,->, font=\scriptsize,>=latex]
          (ka) edge (kb)
           (kb) edge (kc)

            (kd) edge (ke)
            (ke) edge (kf)

              (kb) edge (a1)
              (kb) edge (a2)
               (a1) edge (b1)
   (a1) edge (kd)
   (a2) edge (b2)
  (a2) edge (kd)
  (kc) edge (b1)
  (kc) edge (b2)
  (b1) edge (ke)
  (b2) edge (ke)
   ;
  \end{tikzpicture}
  \caption{Mapping for $1$ negative exponent}\label{check5}
  \end{figure}

  \pagebreak

  \item[$0$ negative exponent] The monomial $X_0^{a_0}\cdot X_1^{a_1}\cdot X_2^{a_2}$ where none of the exponents is negative $a_i>0$, gives the complex Figure \ref{check7}.

  \begin{figure}[H]
  \centering
   \begin{tikzpicture}
   []
          \matrix (m) [
              matrix of math nodes,
              row sep=0.5em,
              column sep=2.5em,
                     ]
  { ~ & ~ &  |[name=a1]|K_0 & |[name=b1]|K_{01} &~&~\\
   ~ & ~ & \oplus & \oplus &~&~\\
   |[name=ka]| 0 &|[name=kb]| K_{H^0} & |[name=kc]|K_1 &|[name=kd]|K_{02} &|[name=ke]|K_{012}&|[name=kf]|0\\
    ~ & ~ & \oplus & \oplus &~&~\\
   ~ & ~ & |[name=a2]| K_2& |[name=b2]|K_{12} &~&~\\
          };

          \path[overlay,->, font=\scriptsize,>=latex]
          (ka) edge (kb)
           (kb) edge (kc)

            (kd) edge (ke)
            (ke) edge (kf)

              (kb) edge (a1)
              (kb) edge (a2)
               (a1) edge (b1)
   (a1) edge (kd)
   (a2) edge (b2)
  (a2) edge (kd)
  (kc) edge (b1)
  (kc) edge (b2)
  (b1) edge (ke)
  (b2) edge (ke)
   ;
  \end{tikzpicture}
  \caption{$0$ negative exponent}\label{check7}
  \end{figure}
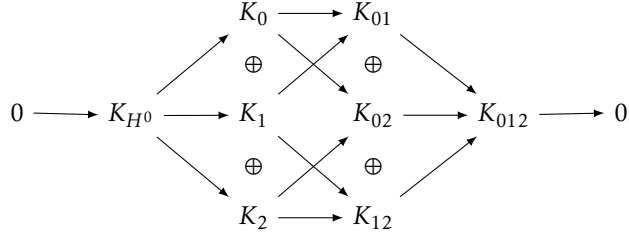

  Consider the SES of complex as in Figure \ref{check8}  . The top and bottom row come from the $1$ negative exponent case, thus giving zero cohomology. The SES of complex gives LES of cohomology groups, since top and bottom row have zero cohomology, so does the middle.

  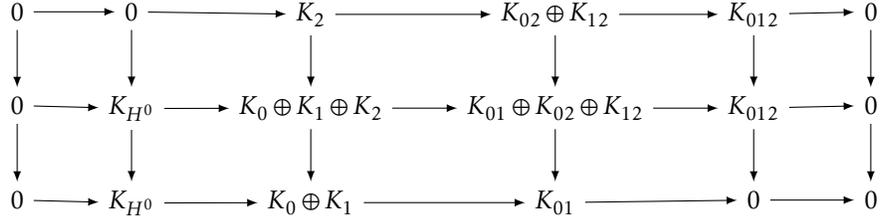
\begin{figure}[H]
  \centering
   \begin{tikzpicture}
   []
          \matrix (m) [
              matrix of math nodes,
              row sep=2em,
              column sep=2.5em,
                     ]
  {   |[name=aa]|0 &  |[name=ab]|0 &  |[name=a1]|K_2 & |[name=b1]|K_{02}\oplus K_{12} &  |[name=c1]|K_{012}& |[name=d1]|0\\
   |[name=ka]| 0 &|[name=kb]| K_{H^0} & |[name=kc]|K_0\oplus K_1\oplus K_2 &|[name=kd]|K_{01}\oplus K_{02} \oplus K_{12} &|[name=ke]|K_{012}&|[name=kf]|0\\
   |[name=qa]| 0 &  |[name=qb]|K_{H^0} & |[name=a2]| K_0\oplus K_1& |[name=b2]|K_{01} & |[name=c2]|0& |[name=d2]|0\\
          };

          \path[overlay,->, font=\scriptsize,>=latex]
  (aa) edge (ab)
  (ab) edge (a1)
  (a1) edge (b1)
  (b1) edge (c1)
  (c1) edge (d1)
          (ka) edge (kb)
          (kb) edge (kc)
          (kc) edge (kd)
          (kd) edge (ke)
          (ke) edge (kf)
  (qa) edge (qb)
  (qb) edge (a2)
  (a2) edge (b2)
  (b2) edge (c2)
  (c2) edge (d2)

  (aa) edge (ka)
  (ab) edge (kb)
  (a1) edge (kc)
  (b1) edge (kd)
  (c1) edge (ke)
  (d1) edge (kf)

          (ka) edge (qa)
          (kb) edge (qb)
          (kc) edge (a2)
          (kd) edge (b2)
          (ke) edge (c2)
  	    (kf) edge (d2)

   ;
  \end{tikzpicture}
  \caption{SES of Complex}\label{check8}
    \end{figure}
  \end{description}

\end{proof}
\subsection{Kunneth Formula}\label{k1} We can produce a complex for $\ds{P}^n\times \ds{P}^m$ by taking tensor product of the corresponding \v{C}ech complex associated with each space, and by the Theorem of Eilenberg-Zilber we get
{\small
\begin{equation}
H^i(\ds{P}^n\times \ds{P}^m,\curly{O}(a,b))=\sum_{j=0}^iH^j(\ds{P}^n,\curly{O}(a))\otimes H^{i-j}(\ds{P}^m,\curly{O}(b))~~a,b\in \ds{Q}
\end{equation}
}
Furthermore, we can define a cup product following \cite[pp 194]{liu2002algebraic} to get a homomorphism
{\small
\begin{equation}
\smile:H^p(\ds{P}^n,\curly{O}(a)) \times H^q(\ds{P}^m,\curly{O}(b))\ra H^{p+q}(\ds{P}^n\times\ds{P}^m,\curly{O}(a,b))~~a,b\in \ds{Q}
\end{equation}

}

\section{Tangent Space}\label{tangent-space}

The tangent space is defined as its analogue in differential geometry.

\begin{define}
  Let $f$ be an irreducible polynomial of rational degree in $k[X_1,\ldots, X_n]_{\ds{Q}}$, then the tangent space at $P=(a_1,\ldots,a_n)\in V(f)$ is defined as
  \begin{equation}
    T_P(V):=\sum_{i=1}^n\frac{\partial f}{\partial X_i}(P)(X_i-a_i)=0.
  \end{equation}
\end{define}
The maximal ideal of the ring $k[X_1,\ldots,X_n]_{\ds{Q}}$ corresponding to a point $(a_1,\ldots,a_n)$ is given as $\id{m}$ below (also shown in remark \ref{maximal-ideal}).

\begin{equation}
  \id{m}:=\{X_1^{1/i}-a_1^{1/i},\ldots,X_n^{1/i}-a_n^{1/i}\}_{i\in\ds{Z}_{>0}},\quad\text{and}\quad \id{m}=\id{m}^2,\qquad\text{implying}\qquad \id{m}/\id{m}^2=0.
\end{equation}
Thus, it is not possible to express the tangent space as dual of $\id{m}/\id{m}^2$.
\subsection{Jacobian Criterion}

If the irreducible hypersurface is generated by $(f_1,\ldots,f_r)\in k[X_1,\ldots,X_n]_{\ds{Q}}$ then the tangent space is given as the cokernel of the map $k^r\ra k^n$, where the map is the Jacobian matrix.
\begin{equation}
  \begin{pmatrix}
    \dfrac{\partial f_1}{\partial x_1}(p)&\cdots &\dfrac{\partial f_r}{\partial x_1}(p)\\
    \vdots&\ddots &\vdots\\
    \dfrac{\partial f_1}{\partial x_n}(p)&\cdots &\dfrac{\partial f_r}{\partial x_n}(p)
  \end{pmatrix}.
\end{equation}

\section{Real Degrees} Every real number can be approximated by a sequence of rational numbers, for example take the decimal expansion for the real number. This allows construction of real degree polynomials obtained by a sequence of rational degree polynomials. Let $j\in\ds{R}_{\geq 0}$ be approximated by a sequence of rational numbers $\{r_1,r_2,\ldots\}$, then the polynomial $X^j-t$ can be obtained as a limit from the sequence of polynomials $X^{r_1}-t, X^{r_2}-t, \ldots$. Hence, it makes perfect sense to talk about polynomials of the form
$X^{\sqrt{2}}-1$.

Let the one variable ring of polynomials with real degree be denoted as $k[X]_{\ds{R}}$, then the basis is uncountable and formed by $\{X^r\}_{r\in\ds{R}_{\geq 0}}$, thus the standard proof of Nullstellensatz does not apply. But, it is still possible to talk about zeros of such a polynomial in terms of affine algebraic sets. For example the zero of $X^{\sqrt{2}}-\sqrt{5}^{\sqrt{2}}$ is $\sqrt{5}$.

The definitions can be carried over word for word, but it is not clear what would be the corresponding theorems. Some results are clear as the day, for example formation of projective space with real degree, the corresponding localization at $\{1,T_i,T_i^2,\ldots\}$ to get to affine space or simply localization at the multiplicatively closed set $\{1,f,f^2,\ldots\}$.
 These real degrees can be immensely useful in mirror symmetry.

\section{Applications}\label{sec-app}

In this section applications of rational degree are discussed. All of the applications are not possible in integer degree.

\subsection{Embeddings into projective space} Innumerable embeddings of rational degree in the projective space can be constructed. An example is given below, the reader is encouraged to come up with their own examples.

\begin{example} The following is a degree one embedding over an algebraically closed field $k$.
\begin{equation}
  \begin{aligned}
    [x:y]&\ra[x:x^{1/2}y^{1/2}:y]\\
    [x:y]&\ra[x:x^{1/3}y^{2/3}:x^{1/2}y^{1/2}:x^{2/3}y^{1/3}:y]\\
    [x:y]&\ra[x:x^{1/4}y^{3/4}:x^{1/2}y^{1/2}:x^{3/4}y^{1/4}:y]\\
    [x:y]&\ra[x:x^{1/5}y^{4/5}:x^{2/5}y^{3/5}:x^{1/2}y^{1/2}:x^{3/5}y^{2/5}:x^{4/5}y^{1/5}:y]\\
    \vdots&\qquad\qquad\qquad\vdots\qquad\qquad\qquad\vdots
      \end{aligned}
\end{equation}
Each space comes out of $[x:y]\in \ds{P}^1$, but an infinite tower of spaces can be constructed on the right hand side with one space sitting inside the other space, just like Matryoshka dolls sitting inside each other.
\end{example}

\subsection{Perfectoid Rings}

Let $K$ be a perfectoid field and $\nfld_K$ be its ring of integers which is $p$ adically complete (as introduced in \cite{scholze_1}). Thus, $p$ adic completion of $\nfld_K$ again gives $\nfld_K$ and the neighborhood
of zero is generated by the prime $p$. For example, $K=\ds{Q}_p(p^{1/p^\infty})\pht$ and $\nfld_k=\ds{Z}_p[p^{1/p^\infty}]\pht=\cup_{i\in\ds{N}}\ds{Z}_p[p^{1/p^i}]\pht$ where $\pht$ denotes the $p$ adic completion.

In this section restricted power series $K\lr{X^{1/p^\infty}}$ are constructed via direct limit as introduced in this manuscript. The original construction as given in \cite{scholze_1} is via inverse limit perfection.

Following the construction \ref{constr-dir-lim-2-var} consider the ring of polynomials with degree $\ds{Z}[1/p]$ and coefficients in $\nfld_K$
\begin{equation}
  \nfld_K[X^{1/p^\infty}]:=\nfld_K[X]_{\ds{Z}[1/p]}=\varinjlim_i\nfld_K[X,\ldots,X^{1/p^i}] .
\end{equation}
The ring $\nfld_K[X^{1/p^\infty}]$ can be completed with respect to prime $p$ to obtain the restricted power series ring $\nfld_K\lr{X^{1/p^\infty}}$. The construction of restricted power series is shown in  \cite[pp 213, \S 4.2]{n1998commutative}.

Similarly, the ring $\nfld_K[X_0^{1/p^\infty},\ldots, X_n^{1/p^\infty}]:=\nfld_K[X_0,\ldots, X_n]_{\ds{Z}[1/p]}$ can be completed $p$ adically (via inverse limit $\varprojlim$) to give
\begin{equation}
  \nfld_K\lr{X_0^{1/p^\infty},\ldots, X_n^{1/p^\infty}}:=\varprojlim\varinjlim_i\nfld_K[X_0,\ldots,X_0^{1/p^i},\ldots,X_n,\ldots X_n^{1/p^i}].
\end{equation}
More concretely the ring $\nfld_K\lr{X_0^{1/p^\infty},\ldots, X_n^{1/p^\infty}}$ consists of power series whose coefficients converge $p$ adically to zero as given below.
\begin{equation}
\sum_{j\in\ds{Z}[1/p]^n_{\geq0}}a_{i_1\ldots i_n}X_0^{i_1}\cdots X_n^{i_n},~~ a_j\in \nfld_K, ~~\lim_{j\ra\infty}\abs{a_j}=0,
\end{equation}
where $j$ is a multi index representing the tuple $(i_1,\ldots, i_n)$. Taking the generic fiber gives the ring $K\lr{X^{1/p^\infty}}=K\otimes_{\nfld_K}{\nfld_K}\lr{X^{1/p^\infty}}=K\lr{X^{1/p^\infty}} $.

 The approach in this paper also allows  construction of $\nfld_K\lr{X}_{\ds{Z}[1/d]}$ or $\nfld_K\lr{X}_{\ds{Q}}$ obtained from $p$ adic completion of $\nfld_k[X]_{\ds{Z}[1/d]}$ and $\nfld_K[X]_\ds{Q}$ respectively. These rings cannot be constructed via inverse limit perfection. Furthermore, taking the generic fiber or inverting $p$ gives
\begin{equation}
  K\otimes_\nfld_K \nfld_K\lr{X}_{\ds{Z}[1/d]} = K\lr{X}_{\ds{Z}[1/d]}\quad\text{ and }\quad K\otimes_\nfld_K \nfld_K\lr{X}_\ds{Q} = K\lr{X}_\ds{Q} .
\end{equation}

The direct limit approach allows the use Grothendieck's and Raynaud's Theory of formal schemes to construct perfectoid spaces, rather than using Huber's approach as done by \cite{scholze_1}. This work is being done in \cite{beditiltmay2019}.

\subsection{Rees Algebra}

Let $\ds{Z}_p$ be the ring of $p$ adic integers, then the Rees Algebra (blow up algebra) associated with the ideal $\lr{p}$ is given as
\begin{equation}
  \bigoplus_{n\geq 0}p^n=\ds{Z}_p\oplus p \oplus p^2\oplus\ldots.
\end{equation}
Reducing the above $\mod p$ destroys the Rees algebra giving only $\ds{F}_p$. This destruction can be avoided if constructions in this manuscript are used.

The ring $\ds{Z}_p[X]_{\ds{Q}}$ gives the ring $\ds{Z}_p[p]_{\ds{Q}}$ by setting $X=p$. Let us now define the modified fractional blow up algebra as
\begin{equation}
\ds{Z}_p\oplus\left(\bigoplus_{n\in\ds{Q}_{>0}}p^n\right).
\end{equation}

Reducing the above $\mod p$ does not destroy the blow up algebra completely. The example above carries verbatim to a Discrete Valuation Ring with uniformizing parameter $\varpi$.

\pagebreak

\bibliographystyle{alpha}
\bibliography{\myreferences}

\end{document}